\documentclass[12pt]{article}
\usepackage{amsmath,amssymb,amsthm,amsfonts,mathrsfs}
\usepackage{appendix}
\usepackage{enumitem,graphics,xcolor,yfonts,colonequals}
\usepackage{stmaryrd}
\usepackage[alphabetic]{amsrefs}
\usepackage[all,cmtip]{xy}
\usepackage{tikz-cd}
\usepackage[colorlinks,anchorcolor=blue,citecolor=blue,linkcolor=blue,urlcolor=blue,bookmarksopenlevel=1]{hyperref}
\usepackage[margin=1in]{geometry}

\usepackage{fancyhdr}
\fancypagestyle{titlepage}
{
	\fancyhf{}

	\fancyfoot[l]{
	\href{https://mathscinet.ams.org/mathscinet/msc/msc2020.html}
		{\emph{2020 Mathematics Subject Classification}}
		14E08, 14J28\\
		\emph{Keywords}: Heegner divisors, irrationality, K3 surfaces, moduli spaces
	}
}

\AtBeginDocument{%
	\def\MR#1{}
}

\newcommand{\bZ}{\mathbb{Z}}    
\newcommand{\bQ}{\mathbb{Q}}    
\newcommand{\bC}{\mathbb{C}}    

\newcommand{\cC}{\mathcal{C}}   
\newcommand{\cF}{\mathcal{F}}   
\newcommand{\cLL}{\mathcal{L}}  
\newcommand{\cM}{\mathcal{M}}  
\newcommand{\cO}{\mathcal{O}}   
\newcommand{\bP}{\mathbb{P}}    

\newcommand{\Hom}{\operatorname{Hom}}   
\newcommand{\Gr}{\mathrm{Gr}}           
\newcommand{\CGr}{\mathrm{CGr}}         
\newcommand{\Pic}{\operatorname{Pic}}	
\newcommand{\GL}{\mathrm{GL}}         
\newcommand{\PGL}{\mathrm{PGL}}         
\newcommand{\SL}{\mathrm{SL}}           

\newcommand{\irr}{\operatorname{irr}}   
\newcommand{\covgon}{\operatorname{cov.gon}}   

\newtheorem*{thm*}{Theorem}
\newtheorem*{prop*}{Proposition}
\newtheorem*{cor*}{Corollary}
\newtheorem{thm}{Theorem}[section]
\newtheorem{prop}[thm]{Proposition}
\newtheorem{cor}[thm]{Corollary}
\newtheorem{lemma}[thm]{Lemma}

\numberwithin{equation}{section}

\theoremstyle{definition}

\newtheorem{rmk}[thm]{Remark}

\begin{document}
\title{On the irrationality of moduli spaces of K3~surfaces}
\author{Daniele Agostini\and
Ignacio Barros\and
Kuan-Wen Lai}
\date{}

\newcommand{\ContactInfo}{{
\bigskip\footnotesize

\bigskip
\noindent D.~Agostini,
\textsc{Universit\"at T\"ubingen\\
Fachbereich Mathematik, Auf der Morgenstelle 10, 72076 Tübingen, Germany}\par\nopagebreak
\noindent\textsc{Email:} \texttt{daniele.agostini@uni-tuebingen.de}\\
\textsc{Max-Planck-Institut f\"{u}r Mathematik in den Naturwissenschaften\\
Inselstrasse 22, 04103 Leipzig, Germany}\par\nopagebreak
\noindent\textsc{Email:} \texttt{daniele.agostini@mis.mpg.de}

\bigskip
\noindent I.~Barros,
\textsc{Departement Wiskunde, Universiteit Antwerpen\\
Middelheimlaan 1, 2020 Antwerp, Belgium}\par\nopagebreak
\noindent\textsc{Email:} \texttt{ignacio.barros@uantwerpen.be}

\bigskip
\noindent K.-W.~Lai,
\textsc{Mathematisches Institut der Universit\"at Bonn\\
Endenicher Allee 60, 53121 Bonn, Germany}\par\nopagebreak
\noindent\textsc{Email:} \texttt{kwlai@math.uni-bonn.de}
}}

\maketitle
\thispagestyle{titlepage}

\begin{abstract}
We study how the degrees of irrationality of moduli spaces of polarized K3 surfaces grow with respect to the genus $g$. We prove that the growth is bounded by a polynomial function of degree $14+\varepsilon$ for any $\varepsilon>0$ and, for three sets of infinitely many genera, the bounds can be refined to polynomials of degree $10$. The main ingredients in our proof are the modularity of the generating series of Heegner divisors due to Borcherds and its generalization to higher codimensions due to Kudla, Millson, Zhang, Bruinier, and Westerholt-Raum. For special genera, the proof is also built upon the existence of K3 surfaces associated Hodge theoretically with certain cubic fourfolds, Gushel--Mukai fourfolds, and hyperk\"{a}hler fourfolds.
\end{abstract}


\section{Introduction}
\label{sect:intro}

The coarse moduli space $\cF_g$ of polarized K3 surfaces of genus $g$ parametrizes isomorphism classes of pairs $(S,H)$ where $S$ is a complex K3 surface and $H$ is a primitive ample line bundle with $(H^2)=d=2g-2$. In analogy with the moduli space of curves $\cM_g$, the space $\cF_g$ is unirational for $g$ small enough: Mukai's celebrated results \cites{Muk88,Muk92,Muk06,Muk10,Muk16} established structure theorems for $\cF_g$ which imply unirationality in the range $g\leq 12$ and $g=13,16,18,20$. Other unirationality results have been obtained by Nuer \cite{Nue17} for $g=14$ and by Farkas and Verra \cites{FV18,FV19} for $g=14,22$.

However, as for $\mathcal{M}_g$, the space $\mathcal{F}_g$ is of general type when $g$ is sufficiently large. More precisely, $\cF_g$ is of general type for $g>62$ and for some isolated values $g>46$, see \cite{GHS07}*{Theorem~1}. The question that we study in this paper is:
$$
    \text{\it How irrational is }\mathcal{F}_g\text{ \it as }g\text{ \it grows?}
$$

In \cite{BPELU17}, various measures of irrationality for an irreducible complex projective variety $X$ are revisited, the most important ones being the \emph{degree of irrationality} and the \emph{covering gonality}. The degree of irrationality, first introduced in \cite{MH82}, is denoted by $\irr(X)$ and defined as the smallest degree of a dominant rational map
$$\xymatrix{
    X\ar@{-->}[r] & \bP^{\,\dim X}.
}$$
The \emph{covering gonality}, denoted by $\covgon(X)$, is the smallest possible gonality of a curve passing through a general point of $X$. This measure was studied initially in \cite{LP95} and later in \cite{Bas12}. These two birational invariants are related by the inequality:
\begin{equation}
\label{eqn:irr-covgon}
    \covgon(X) \leq \irr(X).
\end{equation}
We see that $\irr(X)=1$ if and only if $X$ is rational, and $\covgon(X)=1$ if and only if $X$ is uniruled. Nevertheless, the two invariants capture different phenomena. For instance, if $S$ is a polarized K3 surface of genus $g$, then a classical result of Bogomolov and Mumford \cite{MM83}*{page~351} implies that $S$ can be covered by possibly singular elliptic curves, hence $\covgon(S)=2$, whereas the best known upper bound for the degree of irrationality is $\irr(S)\leq O(\sqrt{g})$ \cite{Sta17}*{Theorem~5.1}.

In general, the class of varieties for which the degree of irrationality or the covering gonality is known is fairly narrow. Donagi proposed to find bounds on these invariants for the moduli space $\mathcal{M}_g$ of genus $g$ curves and the moduli space $\mathcal{A}_g$ of principally polarized abelian varieties of dimension $g$ \cite{BPELU17}*{Problem~4.4}. The only known upper bound on the covering gonality of $\mathcal{M}_g$ is 
given by the Hurwitz number $h_{g,\delta(g)}$ counting simply ramified covers of $\bP^1$ of degree $\delta(g)=\lfloor\frac{g+3}{2}\rfloor$ with $2g+2\delta(g)-2$ fixed branch points. However, these numbers are quite large and grow more than exponentially with $g$. To the best of the authors' knowledge, there is no better upper bound for the covering gonality of $\mathcal{M}_g$, nor any kind of upper bound for the degree of irrationality. As for lower bounds, the only known bound for the covering gonality (and thus for the degree of irrationality) of both $\mathcal{M}_g$ and $\mathcal{F}_g$ for $g$ large is $2$ due to the classical results \cites{HM82,GHS07}. An interesting question is whether these numbers actually grow to infinity with $g$.


As the first approach to the question about measuring the irrationality of moduli spaces, we give some evidence that the situation for $\mathcal{F}_g$ may be substantially simpler than for $\mathcal{M}_g$. The main result in this paper is:

\begin{thm}
\label{thm:allK3}
For every $\varepsilon>0$, there exists a constant $C_\varepsilon>0$ such that
$$
    \irr(\mathcal{F}_g)
    \leq C_\varepsilon\cdot g^{14+\varepsilon}
    \quad\text{for all}\quad
    g.
$$
\end{thm}

Note that this induces the same bound for the covering gonality by \eqref{eqn:irr-covgon}. For infinitely many $g$'s the bound can be improved. Our second result is the following:

\begin{thm}
\label{thm:specialK3}
Let $d\colonequals 2g-2 > 6$, $n$ a positive integer, and assume one of the following:
\begin{enumerate}[label=\textup{(\Alph*)}]
\setlength\itemsep{0pt}
	\item\label{case:cubicFourfold}
	$d\equiv 0$ or $2\;(\bmod\;6)$, and is not divisible by $4$, $9$, or any odd prime $p\equiv 2\;(\bmod\;3)$.
	\item\label{case:Gushel-Mukai}
	$d\equiv 2$ or $4\;(\bmod\;8)$, and is not divisible by any prime $p\equiv 3\;(\bmod\;4)$.
	\item\label{case:specialHK} $\frac{d}{2}-n$ is a square.
\end{enumerate}
Then there exists a constant $C>0$, depending on $n$ in case \ref{case:specialHK}, such that 
$$
    \irr\left(\cF_g\right)
    \leq C\cdot g^{10}.
$$
\end{thm}

To illustrate the idea in our proof, consider case \ref{case:cubicFourfold} of Theorem \ref{thm:specialK3}. Our starting point is the vastly used dictionary \cite{Has00} between K3 surfaces and special cubic fourfolds. Indeed, whenever the hypothesis of \ref{case:cubicFourfold} holds, Hassett proved that there are rational maps
$\mathcal{F}_{g} \dashrightarrow \mathcal{C}$
into the period space $\mathcal{C}$ of cubic fourfolds. These maps are generically finite of degree one or two onto the Heegner divisor $\mathcal{C}_{2g-2}$, which is the moduli spaces of special cubic fourfolds of discriminant $2g-2$. In particular
$$
    \irr(\mathcal{F}_g) \leq 2\cdot \irr(\mathcal{C}_{2g-2}),
$$
and it is sufficient to prove that the right hand side is bounded by a polynomial.

As $\mathcal{C}$ is a quasiprojective variety, the divisors $\mathcal{C}_d$ admit simultaneous embeddings in $\bP^N$ with $N$ sufficiently large. Let $\deg(\mathcal{C}_{d})$ denote the degree of the closure $\overline{\mathcal{C}}_{d}$ in $\mathbb{P}^N$. Then
$$
    \irr(\mathcal{C}_{d})
    \leq\deg(\mathcal{C}_{d})
$$
since a general linear projection $\overline{\mathcal{C}}_{d} \longrightarrow \mathbb{P}^{\,\dim \mathcal{C}_{d}}$ has degree exactly $\deg(\mathcal{C}_{d})$. At this point, a fundamental result of Borcherds \cite{Bor99} , together with the refinement \cite{McG03}, implies that the generating series
$$
    \sum_d \deg(\mathcal{C}_{d}) \cdot q^{\frac{d}{6}}
$$
is a modular form of weight $11$, and classical estimates for the Fourier coefficients of a modular form show that $\deg(\mathcal{C}_d) = O(d^{10})$, which concludes the proof of Theorem~\ref{thm:specialK3}~\ref{case:cubicFourfold}. The other two cases are proved similarly: in case \ref{case:Gushel-Mukai} special cubic fourfolds are replaced by special Gushel-Mukai fourfolds introduced in \cite{DIM15}. Case~\ref{case:specialHK} is instead proved by realizing all $\mathcal{F}_g$'s as divisors inside different period spaces (one for each $n$) associated to hyperk\"{a}hler fourfolds of K3-type. For this we build on the description of the period map for hyperk\"{a}hler fourfolds with a polarization of fixed degree and divisibility developed in \cite{DM19}. 
\vspace{8pt}

The same ideas extend to the proof of our main result, which is the uniform bound obtained in Theorem~\ref{thm:allK3} for all $g$. The point is that we can find rational maps $\mathcal{F}_g \dashrightarrow \mathcal{P}_{\#}$ into a \emph{single} period space $\mathcal{P}_{\#}$ for \emph{all} $\mathcal{F}_g$'s. The catch is that the images of $\mathcal{F}_g$ are not divisors anymore but cycles of higher codimension. However, we can use a generalization of Borcherds' result known as {\textit{Kudla's modularity conjecture}}  \cites{Kud97, Kud04} stating modularity of certain generating series for special cycles on Shimura varieties. For Shimura varieties of $O(m,2)$-type, Kudla's modularity conjecture was proved in full generality by Bruinier and Westerholt-Raum \cite{BW15} building on work of Kudla--Millson \cite{KM90} and W.~Zhang \cite{Zha09}.

We remark that Borcherds' result \cite{Bor99} was used by Maulik and Pandharipande \cite{MP13} to compute degrees of Noether--Lefschetz divisors of concrete families of quasi-polarized K3 surfaces of low degree. Similarly, Li and Zhang compute in \cite{LZ13} the degrees of the closure of preimages of $\cC_d$ under the GIT quotient
$$\xymatrix{
    \pi\colon\bP^{55}\cong\left|\cO_{\bP^5}(3)\right|
    \ar@{-->}[r] &
    \left|\cO_{\bP^5}(3)\right|^{s}\sslash {\PGL}(6).
}$$
Moreover, polynomial bounds analogous to those of Theorem \ref{thm:specialK3} can also be obtained via a concrete construction of O'Grady in another special series of examples. We review this in Section~\ref{sec:ogradyexamples}.

The outline of the paper is as follows: In Section~\ref{sec:modular-forms}, we review some preliminary results, set up necessary notations, and give estimates on the degrees of irrationality for the Heegner divisors. In Sections~\ref{sect:spCubics} and \ref{sect:GM-fourfolds}, we treat the cases of special cubic fourfolds and special Gushel-Mukai fourfolds, respectively, and establish polynomial bounds for the degrees of irrationality of their moduli spaces. In Section~\ref{sect:periodHKs} we recall the main results of \cite{DM19} and we conclude the proof of Theorem~\ref{thm:specialK3}. Finally in Section~\ref{sect:uniformEst} we review Kudla cycles, the modularity statement and prove Theorem~\ref{thm:allK3}

\subsection*{Acknowledgments} We would like to thank Emanuele Macr\`{i} for many helpful conversations and a careful reading of the first draft of this paper. We also thank Davide Lombardo and Ricardo Menares for their help with modular forms. Special thanks go to Daniel Huybrechts for telling us about the results in \cite{Zha09} and \cite{OS18}, allowing us to prove our main theorem. We would also like to thank Georg Oberdieck for pointing out an issue with a lemma on a previous version of this article. We thank Klaus Hulek, Giovanni Mongardi and Kieran O'Grady for helpful conversations about O'Grady's construction. Finally, we would like to thank the anonymous referees for many helpful comments and suggestions. The second and third named authors were supported by the ERC Synergy Grant ERC-2020-SyG-854361-HyperK.

\section{Heegner divisors and their irrationality}
\label{sec:modular-forms}

In this section, we review the definitions of vector-valued modular forms, Heegner divisors, Weil representation, and the main results in \cites{Bor99,McG03}. We end the section by linking the modularity statement with degrees of irrationality for Heegner divisors.

\subsection{Fourier coefficients of scalar modular forms}
\label{subsect:FourierScalarModForm}

We denote by $\mathcal{H}$ the upper half plane in $\mathbb{C}$ and by $\Gamma = \SL_2(\mathbb{Z})$ the modular group with its usual action on $\mathcal{H}$. Recall that this group is generated by
\[
S = \begin{pmatrix}
0 & -1\\
1 & 0
\end{pmatrix}, \quad 
T = \begin{pmatrix}
1 & 1\\
0 & 1        \end{pmatrix}.
\]
Furthermore, for $N\geq 2$ we consider the congruence subgroup of level $N$
\[
\Gamma(N) = \ker\left[
\SL_2(\mathbb{Z})\longrightarrow\SL_2(\mathbb{Z}/N\mathbb{Z}) 
\right].
\]  
Given an integer $k\in\bZ$, a \emph{modular form of weight $k$ for $\Gamma(N)$} is a holomorphic function $F\colon\mathcal{H}\longrightarrow \mathbb{C}$  that satisfies
$$
F(g\cdot \tau) = (c\tau+d)^{k}F(\tau),
\quad\text{ for all }
g = \begin{pmatrix} a & b \\ c & d \end{pmatrix}\in \Gamma(N).
$$
Such a modular form has a power series expansion
\[ F(\tau) = \sum_{m\in \mathbb{Z}} c_m q^{\frac{m}{N}}, \qquad q = e^{2\pi i \tau} \]
and here we will always assume that our forms are holomorphic at infinity, meaning that $c_m=0$ for all $m<0$. The Fourier coefficients of such forms grow at most like a polynomial of order $k-1$. We recall the precise result and its proof for completeness. 

\begin{prop}
	\label{prop:growthscalar}
	Let $F$ be a modular form of weight $k\geq 3$ for the group $\Gamma(N)$. Then the Fourier coefficients $c_m$  satisfy $c_m = O(m^{k-1})$.
\end{prop}
\begin{proof}
	Any such modular form can be written as a sum of a cusp form and an Eisenstein series, see, for example, \cite{DS05}*{Section~5.11}. Hecke proved \cite{Zag123}*{Proposition~8} that the Fourier coefficients of a cusp form grow as $O(m^{\frac{k}{2}})$. On the other hand, one can find, e.g., in \cite{DS05}*{Theorem~4.2.3}, an explicit basis of the subspace of Eisenstein series and then observe that their Fourier coefficients are bounded by
	$C\cdot \sum_{d\mid m} d^{k-1}$ for a certain constant $C>0$. At this point, it is enough to observe that
	\[
	m^{k-1} \leq \sum_{d\mid m}d^{k-1} = \sum_{d\mid m} \left( \frac{m}{d} \right)^{k-1} \leq \sum_{d=1}^{\infty} \left( \frac{m}{d} \right)^{k-1} = \zeta(k-1)\cdot m^{k-1}
	\]
	where $\zeta$ is the Riemann zeta function.
\end{proof}

This can be generalized to vector-valued modular forms.

\subsection{Vector-valued modular forms}


Let $V$ be a finite dimensional $\bC$-vector space and let $\rho\colon \Gamma \to \GL(V)$ be a representation of the modular group on $V$ that factors through a finite quotient $\SL_2(\mathbb{Z}/N\mathbb{Z})$. Given an integer $k\in\bZ$, one defines a \emph{modular form of weight $k$ and type $\rho$} as a holomorphic function $F\colon\mathcal{H}\longrightarrow V$  that satisfies
$$
    F(g\cdot \tau) = (c\tau+d)^{k}\rho(g)(F(\tau)),
    \quad\text{ for all }
    g = \begin{pmatrix} a & b \\ c & d \end{pmatrix}\in \Gamma.
$$
Since $\rho$ factors through $\SL_2(\mathbb{Z}/N\mathbb{Z})$ we have a Fourier series expansion
$$
    F = F_1(q)v_1 + \dots + F_n(q)v_n,
    \quad
    F_i(q) = \sum_{m\in\bZ} c_{i,m}\cdot q^{\frac{m}{N}}.
$$
where $\{v_1,\dots,v_n\}$ is a basis of eigenvectors for $\rho(T)$.  Again, we will suppose that the modular forms are holomorphic at infinity, meaning that $c_{i,m} = 0$ for all $m<0$, and then we have the same estimate as before:

\begin{cor}\label{cor:growthvector}
The Fourier coefficients of $F$ satisfy $c_{m,i} = O(m^{k-1})$ for all $i=1,\dots,n$.
\end{cor}
\begin{proof}
	The representation $\rho$ factors through $\SL_2(\mathbb{Z}/N\mathbb{Z})$ exactly when it is trivial on $\Gamma(N)$. Then we see that each of the  $F_i(q)$ is a scalar modular form of weight $k$ for  $\Gamma(N)$, so that the estimate follows from Proposition~\ref{prop:growthscalar}.
\end{proof}

\subsection{Heegner divisors}
\label{subsect:Heegner}

Let $M$ be a lattice of signature $(m,2)$, where $m$ is even, and let $M^\vee\colonequals\Hom(M,\bZ)$ denote its dual lattice. Thanks to the intersection pairing $\langle\,\cdot\,, \,\cdot\,\rangle$ on $M$, we can identify $M^{\vee}$ with the subset $M^{\vee} = \{ v \in M\otimes \mathbb{Q} \,|\, \langle  v , m \rangle \in \mathbb{Z} \text{ for all } m\in M \}$ of $M\otimes \mathbb{Q}$. In particular, we have a natural embedding $M\hookrightarrow M^\vee$. The quotient 
$D(M):=M^\vee\big/M$
is a finite abelian group called the \emph{discriminant group of $M$}, which is equipped with a $\bQ\big/\bZ$-valued pairing induced from the one on $M$. The \emph{monodromy group of $M$}, defined as
$$
    \Gamma_M\colonequals\left\{
        g\in\mathrm{O}(M)
        \mid
        g\text{ acts trivially on }D(M)
    \right\},
$$
acts on the period domain
\begin{equation}
\label{eqn:periodDomain}
    \Omega(M)\colonequals\left\{
        w\in\bP(M\otimes \bC)
        \mid
        \langle w,w\rangle=0, \langle w,\overline{w}\rangle<0
    \right\}.
\end{equation}
Baily and Borel \cite{BB66} proved that the quotient
\begin{equation}
\label{eqn:periodSpace}
    \mathcal{P}_M\colonequals\Omega(M)/\Gamma_M
\end{equation}
is a quasi-projective variety, which admits a normal compactification $\overline{\mathcal{P}}_M$ such that the boundary $\overline{\mathcal{P}}_M\setminus \mathcal{P}_M$ has codimension at least two.

Some natural divisors on $\mathcal{P}_M$ can be described in terms of the lattice as follows: for any $v\in M^{\vee}$, we set
$$
    v^{\perp} = \{
        w\in\Omega(M) \mid \langle v, w\rangle=0
    \}.
$$
Then, for any $n\in\bQ_{>0}$ and $\gamma\in D(M)$, the group $\Gamma_M$ acts on vector in the equivalence class of $\gamma$ with fixed norm $\frac{1}{2}\langle v,v\rangle=n$ with finitely many orbits. The cycle
$$
    \sum_{
        \frac{1}{2}\langle v,v\rangle=n,\;
        v\equiv\gamma
    }
    v^\perp\subset \Omega(M)
$$
is $\Gamma_M$-invariant and descends to a Cartier divisor $Y_{n,\gamma} \subset\mathcal{P}_M$ called a {\it{Heegner divisor}}. They are in general neither reduced nor irreducible, for instance when $\gamma$ is nontrivial and $\gamma=-\gamma$ in $D(M)$, every component of $Y_{n,\gamma}$ has multiplicity two, see for instance \cite{BM19}*{Lemma 4.2}. We denote the class of this divisor by $[Y_{n,\gamma}]\in{\rm{Pic}}(\mathcal{P}_M)$. We observe that $[Y_{n,\gamma}]\neq 0$ if and only if $n \in \frac{1}{N}\mathbb{Z}$ where $N$ is the smallest positive integer such that $\frac{1}{2}\langle \gamma, \gamma \rangle \in \frac{1}{N}\mathbb{Z}$ for all $\gamma \in D(M)$. Moreover, the restriction to $\Omega(M)$ of the tautological bundle $\mathcal{O}(-1)$ on $\bP(M\otimes \bC)$ admits a compatible action of $\Gamma_M$ and descends to a line bundle $\cLL$ on the quotient. When $n=0$, one defines $[Y_{0,0}]$ to be $\cLL^\vee$ and $[Y_{0,\gamma}]=0$ for $\gamma\neq 0$.

Now we can relate Heegner divisors to modular forms. Given a lattice $M$ of signature $(m,2)$ as above, we denote by $v_\gamma$ the element in the group algebra $\bC[D(M)]$ that corresponds to $\gamma\in D(M)$. There is a distinguished representation $\rho_M$ of $\Gamma$ on $\bC[D(M)]$, called the \emph{Weil representation}, which acts on the generators as follows:
\begin{equation}
\label{eqn:Weil_Representation}
    \rho_M(S)(v_\gamma)
    =\frac{\sqrt{i}^{2-m}}{\sqrt{\left|D(M)\right|}}\sum_{\delta\in D(M)}e^{-2\pi i\langle\gamma,\delta\rangle}v_\delta,
    \quad
    \rho_M(T)(v_\gamma)
    = e^{2\pi i\frac{\langle \gamma,\gamma\rangle}{2}}v_\gamma.
\end{equation}
See also \cite{McG03}*{Section~2.3} for more details. We see that the number $N$ defined before is also the smallest positive integer such that $\rho_M(T^N)=\operatorname{Id}$. In fact, the Weil representation factors through  $\SL_2(\mathbb{Z}/N\mathbb{Z})$.

\begin{thm}\cites{Bor99, McG03}
\label{Thm.Bor99}
The generating series
$$
    \Phi(q) = \sum_{\gamma\in D(M)} \Phi_{\gamma}(q) v_\gamma
    \quad\text{where}\quad
    \Phi_{\gamma}(q) = \sum_{n\in \frac{1}{N}\mathbb{Z}^{\geq 0}} [Y_{n,\gamma}] q^n
$$
is a modular form of weight $1+\frac{m}{2}$, type $\rho_M$, and with coefficients in  ${\rm{Pic}}(\mathcal{P}_M)$.
\end{thm}

This theorem means that for every homomorphism $\ell\colon \operatorname{Pic}(\mathcal{P}_M) \longrightarrow \mathbb{C}$, the contraction
$$
    \Theta(q)=\sum_{\gamma\in D(M)} \Theta_{\gamma}(q) v_\gamma \quad\text{where}\quad
    \Theta_{\gamma}(q)=\sum_{n\in \frac{1}{N}\mathbb{Z}^{\geq 0}} \ell([Y_{n,\gamma}]) q^n
$$
is a vector-valued modular form of weight $1+\frac{m}{2}$ and type $\rho_M$.

\begin{rmk}
As mentioned in \cite{LZ13}*{Remark~3}, the lattice $M$ has signature $(2,m)$ in Borcherd's setting \cite{Bor99}, where the generating series of Heegner divisors are modular forms of type $\rho_M^*$, the dual of $\rho_M$. For a lattice $M$ of signature $(m,2)$, we can apply Borcherd's result to $-M$, which has signature $(2,m)$, and use the fact that $\rho_M^*\cong\rho_{-M}$.
\end{rmk}

\subsection{On the irrationality of Heegner divisors}
\label{sec:irrationalheegner}

We can use Borcherds' result to bound the degree of irrationality of the Heegner divisors. Recall that the Baily--Borel compactification $\overline{\mathcal{P}}_M$ is normal and the boundary $\overline{\mathcal{P}}_M\setminus\mathcal{P}_M$ has codimension at least two. In particular, the divisor class groups of $\mathcal{P}_M$ and $\overline{\mathcal{P}}_M$ coincide, so that we have a canonical homomorphism $\Pic(\mathcal{P}_M) \longrightarrow \operatorname{Cl}(\overline{\mathcal{P}}_M)$ which associates to an effective Cartier divisor class $[D]$ on ${\mathcal{P}_M}$ the class of the closure $[\overline{D}]$ inside $\overline{\mathcal{P}}_M$. Now let us fix an embedding $\overline{\mathcal{P}}_M \hookrightarrow \mathbb{P}^N$.
This induces a group homomorphism
$\deg\colon\operatorname{Cl}(\overline{\mathcal{P}}_M)\longrightarrow\bZ$
which associates to each effective Cartier divisor class $[D]$ on $\mathcal{P}_M$ the degree of its closure $\deg(\overline{D})$ inside $\bP^{N}$. From Theorem~\ref{Thm.Bor99}, we obtain:

\begin{cor}
\label{cor.modular}
Let $\overline{Y}_{n,\gamma}$ be the closure of the Heegner divisor $Y_{n,\gamma}$ in $\overline{\mathcal{P}}_M$. The series
$$
    \Theta(q)=\sum_{\gamma\in D(M)} \Theta_{\gamma}(q) v_\gamma \quad\text{where}\quad
    \Theta_{\gamma}(q)=\sum_{n\in \frac{1}{N}\mathbb{Z}^{\geq 0}} \deg(\overline{Y}_{n,\gamma}) q^n
$$
	is a vector-valued modular form of weight $1+\frac{m}{2}$ and type  $\rho_M$. 
\end{cor}

As outlined in the introduction we can use estimates on the coefficients of modular forms to obtain a bound on the degree of irrationality.

\begin{thm}
\label{cor:heegnerirrationality}
Assume $m\geq 4$. Let $Y_{n,\gamma}'$ be an irreducible component of the divisor $Y_{n,\gamma}$. Then the degree of irrationality of $Y_{n,\gamma}'$ satisfies
$$
    \operatorname{irr}(Y_{n,\gamma}') = O(n^{\frac{m}{2}})
$$
\end{thm}

\begin{proof}
Under the embedding in $\mathbb{P}^N$ that we are considering, a general linear projection $\overline{Y}_{n,\gamma}\longrightarrow\bP^{m-1}$ has degree $\deg(\overline{Y}_{n,\gamma})$, whose restriction to $\overline{Y_{n,\gamma}'}$ is dominant of degree at most $\deg(\overline{Y}_{n,\gamma})$. Thus $\irr(Y_{n,\gamma}') \leq \deg(\overline{Y_{n,\gamma}})$ and the result follows from Corollaries~\ref{cor:growthvector} and \ref{cor.modular}.
\end{proof}

\section{Special cubic fourfolds}
\label{sect:spCubics}

In this section, we start with a brief review of special cubic fourfolds and how their moduli spaces appear as Heegner divisors in the period space. We will refine the estimate in Corollary~\ref{cor:heegnerirrationality} for the irrationality of these divisors and then derive Theorem~\ref{thm:specialK3}~\ref{case:cubicFourfold} as a corollary.

\subsection{Special cubic fourfolds and Heegner divisors}

Let $X\subset\bP^5$ be a smooth cubic hypersurface, i.e., a cubic fourfold, and let $h\in H^2(X,\bZ)$ be the class of hyperplane sections. For a very general $X$, the Hodge lattice $H^{2,2}(X,\bZ)$ is spanned by the class $h^2$. We say $X$ is \emph{special} if $H^{2,2}(X,\bZ)$ contains a \emph{labelling}, i.e., a primitive rank-two sublattice $K$ that contains $h^2$. Since the integral Hodge conjecture is valid for cubic fourfolds \cite{Voi13}, the last condition is equivalent to the existence of an algebraic surface $S\subset X$ not homologous to a complete intersection.

The primitive part $\langle h^2\rangle^\perp\subset H^4(X,\bZ)$ equipped with the intersection product is a lattice of signature $(20,2)$, which is isometric to 
$$
    \Lambda_{\rm C}\colonequals
    A_2 \oplus E_8^{\oplus 2}\oplus U^{\oplus 2}.
$$
For every positive integer $d$, we consider the union of hyperplane sections inside $\mathbb{P}(\Lambda_C\otimes \mathbb{C})$
$$
    \widetilde{\cC}_d\colonequals
    \bigcup_{\substack{
        K:\text{ labelling of}\\
        \text{discriminant }d
    }}
    K^\perp\cap\Omega(\Lambda_{\rm C})
$$
and let $\cC_d\subset\mathcal{P}_{\Lambda_{\rm C}}$ be the corresponding $\Gamma_{\Lambda_{\rm C}}$-quotient. By work of Hassett \cite{Has00}, each $\cC_d$ is irreducible, and is nonempty if and only if $d\equiv 0,2\;(\bmod\;6)$. For $d = 2,6$, the loci $\cC_2$ and $\cC_6$ parametrize singular cubics with different types of singularities. For $d>6$, the locus
$
    \cC_d\setminus(\cC_2\cup\cC_6)
$
can be identified as the moduli space of special cubic fourfolds of discriminant $d$ due to the Torelli theorem for cubic fourfolds \cites{Voi86,Laz09}.

The loci $\mathcal{C}_d$ can also be considered as components of Heegner divisors. Indeed, we can denote the elements of $D(\Lambda_{\rm C})=D(A_2)\cong\bZ/3\bZ$ as $\gamma_0$, $\gamma_1$, and $\gamma_2$ such that
$$
    \frac{1}{2}\langle\gamma_i,\gamma_i\rangle
    \equiv \frac{i^2}{3}\;(\bmod{\;\bZ}).
$$
Note that $\gamma_2 = -\gamma_1$. Then a straightforward computation (see, for example, the proof of \cite{LZ13}*{Lemma~1}) shows that
\begin{equation}
\label{eq:cdtoheegner}
    \mathcal{C}_d \subset Y_{\frac{d}{6},\gamma},
    \quad\text{where}\quad
    \gamma = \begin{cases}
        \gamma_0\text{ if }d\equiv0\;(\bmod\;6),
        \\[5pt]
        \gamma_1\text{ if }d\equiv2\;(\bmod\;6).
    \end{cases}
\end{equation}
 Recall that the divisor $Y_{\frac{d}{6},\gamma}$ may possibly contain more than one component.

\begin{cor}\label{cor:irratcd}
	The degrees of irrationality of the $\mathcal{C}_d$ grow at most like a polynomial of degree 10 in $d$:
	\[
		\irr(\mathcal{C}_d) = O(d^{10}).
	\]
\end{cor}
\begin{proof}
 This follows immediately from Theorem~\ref{cor:heegnerirrationality} and the inclusion of \eqref{eq:cdtoheegner}.
\end{proof}

\begin{rmk}
Let $\hat{\cC}_d\subset|\cO_{\bP^5}(3)|$ denote the preimage of $\cC_d\subset\mathcal{P}_{\Lambda_{\rm C}}$ under the GIT~quotient
$
    |\cO_{\bP^5}(3)|^{s}\to
    |\cO_{\bP^5}(3)|^{s}\sslash \PGL(6)
$
followed by the period map. Then the same argument together with the result in \cite{LZ13} give us the same bounds for $\irr(\hat{\cC}_d)$. By choosing a trivializing open subset, one sees that $\hat{\cC}_d$ is birational to $\cC_d\times\bP^{35}$. Nevertheless, comparing $\irr(\cC_d)$ with $\irr(\cC_d\times \bP^{35})$ is a subtle and difficult matter related to the difference between rationality and stable rationality. By choosing the degree induced by a Baily--Borel embedding we circumvent this difficulty with the downside of losing control over an explicit expression for $\Theta(q)$ as in \cite{LZ13}.
\end{rmk}

\subsection{K3 surfaces associated with special cubic fourfolds}
\label{subsect:assoK3_spCubic}

Let us briefly review the notation of associated K3 surfaces and refer the reader to \cite{Has16}*{Section~3} for the details and further references. Let $X$ be a special cubic fourfold marked by a labelling $K\subset H^4(X,\bZ)$ of discriminant $d$. A polarized K3 surface $(S,H)$ of degree $d$ is \emph{associated with} $(X,K)$ if there is an isometry
$$\xymatrix{
    K^{\perp H^4(X,\bZ)}\ar[r]^-\sim &
    H^2(S,\bZ)_{\rm prim}(-1)
}$$
preserving the Hodge structures.

The existence of an associated K3 surface depends only on the discriminant. More precisely, a special cubic fourfold $X\in\cC_d$ admits an associated K3 surface if and only if $d$ is not divisible by $4$, $9$, or any odd prime $p\equiv 2\;(\bmod\;3)$. For a general $X\in\cC_d$, there exists only one associated K3 surface if $d\equiv 2\;(\bmod\;6)$ and exactly two if $d\equiv 0\;(\bmod\;6)$. This induces a dominant rational map
\begin{equation}
\label{eqn:Fg-Cd}
\xymatrix{
    \mathcal{F}_g \ar@{-->}[r] &
    \mathcal{C}_{d}
}
\end{equation}
which has degree either one or two.
We are now ready to prove Case~\ref{case:cubicFourfold} of Theorem~\ref{thm:specialK3}.

\begin{cor}
\label{cor:case_spCubic}
Let $d = 2g-2 > 6$ and suppose that $d\equiv 0$ or $2\;(\bmod\;6)$, and is not divisible by $4$, $9$, or any odd prime $p\equiv 2\;(\bmod\;3)$. Then there exists a constant $C>0$ such that 
$$
    \irr\left(\cF_g\right)
    \leq C\cdot g^{10}.
$$
\end{cor}

\begin{proof}
In this case, the existence of the rational map \eqref{eqn:Fg-Cd} implies that
$$
    \irr(\cF_g)
    \leq 2\cdot\irr(\cC_{d}).
$$
Then the conclusion follows from Corollary \ref{cor:irratcd}.
\end{proof}

\section{Special Gushel--Mukai fourfolds}
\label{sect:GM-fourfolds}

In this section, we start by reviewing the definition of Gushel--Mukai fourfolds and their periods, which was developed in \cites{DIM15, DK19}. We will prove that the loci of special Gushel--Mukai fourfolds in the period space are components of Heegner divisors. Based on this, we will be able to refine the estimate in Corollary~\ref{cor:heegnerirrationality} and derive Theorem~\ref{thm:specialK3}~\ref{case:Gushel-Mukai} as a corollary.

\subsection{Preliminaries on Gushel--Mukai fourfolds}
\label{subsect:preliminary_GM}

Let $V_5$ be a $5$-dimensional vector space over $\bC$ and let
$
    \CGr(2,V_5)
    \subset\bP\left(\bC\oplus\bigwedge^2 V_5\right)
    \cong\bP^{10}
$
be the cone over the Grassmannian $\Gr(2,V_5)$ under the Pl\"{u}cker embedding. A \emph{Gushel--Mukai fourfold} is a transverse intersection
$$
    X = \CGr(2,V_5)\cap Q \cap L\subset \bP^{10},
$$
where $Q$ is a quadric and $L\cong\bP^8$ is a linear subspace. The projection from the vertex of $\CGr(2,V_5)$ defines the \emph{Gushel map}
$p\colon X\longrightarrow\Gr(2,V_5)$.
The pullback $H\colonequals p^*\cO_{\Gr(2,V_5)}(1)$ of the Pl\"{u}cker polarization is ample and $-K_X = 2H$, which makes $X$ a Fano fourfold of degree $10$ and index $2$.

The cohomology group $H^4(X,\bZ)$ endowed with the intersection product is isometric to
$
    \langle1\rangle^{\oplus 22}
    \oplus
    \langle-1\rangle^{\oplus 2}
$
as abstract lattices. It contains the  sublattice
$\Lambda_{\operatorname{Gr}} = p^*H^4(\Gr(2,V_5),\bZ)$, isomorphic to $A_1^{\oplus 2}$, 
whose orthogonal complement in $H^4(X,\bZ)$ is called the \emph{vanishing lattice} of $X$. The vanishing lattice has signature $(20,2)$ and is isometric to
$$
    \Lambda_{\rm GM}\colonequals
    A_1^{\oplus 2} \oplus E_8^{\oplus 2}
    \oplus U^{\oplus 2}
$$
The associated $\Gamma_{\Lambda_{\rm GM}}$-quotient $\mathcal{D} = \mathcal{P}_{\Lambda_{\rm GM}}$ is a $20$-dimensional irreducible quasi-projective variety \cite{BB66}.

\subsection{Heegner divisors of special GM-fourfolds}
\label{subsect:spGM_Heegner}

For every Gushel--Mukai fourfold $X$, the Hodge lattice $H^{2,2}(X,\bZ) 
$
contains $\Lambda_{\Gr}$
as a sublattice by construction. These two lattices coincide for a very general $X$. A Gushel--Mukai fourfold $X$ is called \emph{special} if $H^{2,2}(X,\bZ)$ contains a \emph{labelling}, i.e., a primitive rank-three sublattice $K$ that contains $\Lambda_{\Gr}$. By \cite{DIM15}*{Lemma~6.1}, a labelling $K$ has discriminant $d\equiv 0$, $2$, or $4\;(\bmod{\;8})$.

The monodromy group $\Gamma_{\Lambda_{\rm GM}}$ acts on the set of labellings of discriminant $d$. If $d\equiv 0$ or $4\;(\bmod{\;8})$, then \cite{DIM15}*{Proposition~6.2}, shows that there is only one orbit under this action and the representatives can be taken as $K_d=\langle h_1,h_2,\zeta \rangle$, where $h_1,h_2$ are generators of $\Lambda_{\rm Gr}$ and the intersection product is given by
\begin{equation}\label{eq:GMlattices1}
    K_d \cong\begin{pmatrix}
        2 & 0 & 0\\
        0 & 2 & 0\\
        0 & 0 & \frac{d}{4}
    \end{pmatrix}
    \text{ for }d\equiv 0\;(\bmod{\;8}),
    \quad
    K_d \cong\begin{pmatrix}
        2 & 0 & 1\\
        0 & 2 & 1\\
        1 & 1 & \frac{d+4}{4}
    \end{pmatrix}
    \text{ for }d\equiv 4\;(\bmod{\;8}).
\end{equation}
On the other hand, there are exactly two orbits when $d\equiv 2\;(\bmod{\;8})$. Examples of representatives for the two orbits are $K'_d= \langle h_1,h_2,\zeta' \rangle$ and $K''_d= \langle h_1,h_2,\zeta'' \rangle$ with intersection product given by
\begin{equation}\label{eq:GMlattices2}
    K_d'\cong\begin{pmatrix}
        2 & 0 & 1\\
        0 & 2 & 0\\
        1 & 0 & \frac{d+2}{4}
    \end{pmatrix}
    \quad\text{and}\quad
    K_d''\cong\begin{pmatrix}
        2 & 0 & 0\\
        0 & 2 & 1\\
        0 & 1 & \frac{d+2}{4}
    \end{pmatrix}.
\end{equation}
The period domains for special Gushel--Mukai fourfolds of discriminant $d$ are given by $\widetilde{\mathcal{D}}_{d}$,$\widetilde{\mathcal{D}}'_{d}$ and $\widetilde{\mathcal{D}}''_{d}$, where $\widetilde{\mathcal{D}}_{d}
\colonequals\left\{
w\in \Omega(\Lambda_{\rm GM}) \mid w\perp K_d
\right\}$
for $d\equiv 0,4\text{ (mod }8)$,
and $\widetilde{\mathcal{D}}'_{d},\widetilde{\mathcal{D}}''_{d}$ are defined analogously in terms of $K_d',K_d''$ when $d\equiv 2\text{ (mod }8)$.
Their respective $\Gamma_{\Lambda_{\rm GM}}$-quotients form irreducible divisors $\mathcal{D}_d$, $\mathcal{D}_d'$, and $\mathcal{D}_d''$ in the period domain $\mathcal{D}$.

Let us fix some notations before discussing the relations between these divisors and the Heegner divisors associated with $\mathcal{D}$. Suppose that the component $A_1^{\oplus 2}$ of $\Lambda_{\rm GM}$ is spanned by the elements $e$ and $f$ where $\langle e,e \rangle = \langle f, f \rangle = 2$ and $\langle e, f \rangle = 0$. Then their dual elements $e_*$ and $f_*$ in $(A_1^{\oplus 2})^\vee$ satisfy $\langle e_*,e_* \rangle = \langle f_*,f_*\rangle = 1/2$ and $\langle e_*, f_* \rangle= 0$. We will fix the isomorphism between discriminant groups:
\begin{equation}
\label{eqn:disc_2(A2)}
    D\left(\Lambda_{\rm GM}\right)
    \cong D\left(A_1^{\oplus 2}\right)
    =\{0,e_*,f_*,e_*+f_*\}.
\end{equation}

\begin{lemma}
\label{lemma:HeegnerGM}
The period domains $\mathcal{D}_d$, $\mathcal{D}_d'$, and $\mathcal{D}_d''$ are components of Heegner divisors in the following way:
$$\begin{array}{ccl}
    \mathcal{D}_d \subset Y_{\frac{d}{8}, 0},
    & \text{if} & d\equiv 0\;(\bmod{\;8})
    \\[10pt]
     \mathcal{D}_d'\subset Y_{\frac{d}{8}, e_*} \quad\text{and}\quad \mathcal{D}_d''\subset Y_{\frac{d}{8}, f_*}
    & \text{if} & d\equiv 2\;(\bmod{\;8})
    \\[10pt]
    \mathcal{D}_d\subset Y_{\frac{d}{8}, e_*+f_*} 
    & \text{if} & d\equiv 4\;(\bmod{\;8})
\end{array}$$
\end{lemma}

\begin{proof}
For each $v\in\Lambda_{\rm GM}^\vee$, one can verify that
$$
    8n
    = 8\cdot\frac{1}{2}\langle v,v\rangle
    = 4\langle v,v\rangle
    \equiv \begin{cases}
    0\;(\bmod{\;8}) &\text{if } v\equiv 0\;(\bmod{\;\Lambda_{\rm GM}}),
    \\[5pt]
    2\;(\bmod{\;8}) &\text{if } v\equiv e_*\text{ or }f_*\;(\bmod{\;\Lambda_{\rm GM}}),
    \\[5pt]
    4\;(\bmod{\;8}) &\text{if } v\equiv e_*+f_*\;(\bmod{\;\Lambda_{\rm GM}}).
    \end{cases}
$$

Let us first assume that $d\equiv 0\;(\bmod{\;8})$. We would like to prove that $\mathcal{D}_d$ is a component of $Y_{\frac{d}{8},0}$. For the inclusion $\mathcal{D}_d \subset Y_{\frac{d}{8},0}$, if $K_d=\langle h_1,h_2,\zeta \rangle$ is a labelling as in \eqref{eq:GMlattices1}, we need to prove that $K_d^{\perp}\cap \Omega(\Lambda_{\rm GM})$ is of the form $v^{\perp}$ for one $v\in \Lambda_{\rm GM}^{\vee}$ such that $\frac{1}{2}\langle  v,v\rangle= \frac{d}{8}$ and $v\equiv 0\;(\bmod{\;\Lambda_{\rm GM}})$. This can be done  simply by taking $v=\zeta$: indeed we see that $\langle v,h_1 \rangle = \langle v,h_2 \rangle = 0$ so that $v\in \Lambda_{\rm GM}^{\perp}$, and moreover $\frac{1}{2}\langle v,v \rangle = \frac{d}{8}$ so that $v\equiv 0\; (\bmod{\; \Lambda_{{\rm{GM}}}})$. This shows that $\mathcal{D}_d\subset Y_{\frac{d}{8},0}$ and for dimension reasons it must be a component.

The proof in the case $d\equiv 4\; (\bmod{\; 8})$ is similar: given the lattice $K_d=\langle h_1,h_2,\zeta \rangle$ as in \eqref{eq:GMlattices1}, we can take the vector $v=\zeta-\frac{1}{2}h_1-\frac{1}{2}h_2$. 

Now assume that $d\equiv 2\;(\bmod{\; 8})$. If we take the two lattices $K'_d=\langle h_1,h_2,\zeta' \rangle$ and $K''_d=\langle h_1,h_2,\zeta'' \rangle$ as in \eqref{eq:GMlattices2}, we can define the two vectors $v' = \zeta'-\frac{1}{2}h_1$ and $v''=\zeta'-\frac{1}{2}h_2$. We see as above that $v',v''\in \Lambda^{\vee}_{\rm GM}$ and $\frac{1}{2}\langle v',v' \rangle = \frac{1}{2}\langle v'',v'' \rangle = \frac{d}{8}$. Moreover, $v',v''$ are not equivalent modulo $\Lambda_{\rm GM}$ so, up to exchanging $e_*,f_*$ we can assume that $v'\equiv e_*$ and $v''\equiv f_*$ modulo $\Lambda_{\rm GM}$. This shows that $\mathcal{D}'_d \subset Y_{\frac{d}{8},e_*}$ and $\mathcal{D}''_d \subset Y_{\frac{d}{8},f_*}$. 

\end{proof}

\begin{cor}\label{cor:irratdd}
	The degrees of irrationality of $\mathcal{D}_d,\mathcal{D}'_d,\mathcal{D}''_d$ grow at most like a polynomial of degree 10 in $d$:
	\[
	\irr(\mathcal{D}_d), \; \irr(\mathcal{D}'_d),\; \irr(\mathcal{D}''_d) = O(d^{10}).
	\]
\end{cor}
\begin{proof}
This follows immediately from Theorem~\ref{cor:heegnerirrationality} and the identification of Lemma \eqref{lemma:HeegnerGM}, once we recall that these period domains are all irreducible.
\end{proof}

\begin{rmk}
Consider the series for the degrees of Heegner divisors associated to special Gushel--Mukai fourfolds
\begin{equation}
\label{eqn:rmkGM}\Theta(q)=\sum_{n,\gamma}\deg(\overline{Y}_{n,\gamma})q^n.
\end{equation} 
By looking at the eigenvectors of the Weil representation evaluated at the generators of the modular group $\Gamma$, one can place 
\eqref{eqn:rmkGM} in a much smaller vector space of modular forms. More concretely, \eqref{eqn:rmkGM} is a scalar-valued modular form for the group 
$$
    \Gamma^0(4)
    \colonequals\left\{
        \begin{pmatrix}
            a & b \\
            c & d
        \end{pmatrix}\in\SL_2(\mathbb{Z})
        \,\bigg|\,
        b \equiv 0\;(\bmod{\;4})
    \right\}
$$
and character $\chi_{-4}$ determined by
$$
    \chi_{-4}\left(\begin{array}{cc}a&b\\c&d\end{array}\right)=\chi_{-4}(d),
    \quad\text{ for all } \left(\begin{array}{cc}a&b\\c&d\end{array}\right)\in \Gamma^0(4),
$$
where $\chi_{-4}$ is the unique nontrivial Dirichlet character modulo $4$.

\end{rmk}

\subsection{K3 surfaces associated with special GM-fourfolds}
\label{subsect:assoK3_spGM}

A picture analogous to the case of special cubic fourfolds holds for special Gushel--Mukai fourfolds \cite{DIM15}*{Section~6.2}. Let $X$ be a special Gushel--Mukai fourfold marked by a labelling $K\subset H^4(X,\bZ)$ of discriminant $d=2g-2$. Then a polarized K3 surface $(S,H)$ of genus $d$ is \emph{associated with} $(X,K)$ if there is an isometry
$$\xymatrix{
    K^{\perp H^4(X,\bZ)}\ar[r]^-\sim &
    H^2(S,\bZ)_{\rm prim}(-1)
}$$
preserving the Hodge structures. There exists such a K3 surface if and only if the discriminant $d$ is congruent to either $2$ or $4$ modulo $8$ and it is not divisible by any prime $p\equiv 3\;(\bmod\;4)$. Moreover, the associated K3 surface is generically unique and this induces birational maps
\begin{equation}
\label{eqn:Fg-Dd}
\begin{gathered}
\xymatrix{
    \mathcal{F}_g \ar@{-->}[r]^-\sim &
    \mathcal{D}_{d}
}
\quad\text{for}\quad
d\equiv 4 \;(\bmod\;8), \\
\xymatrix{
    \mathcal{F}_g \ar@{-->}[r]^-\sim &
    \mathcal{D}_d'
} 
\quad\text{and}\quad
\xymatrix{
    \mathcal{F}_g \ar@{-->}[r]^-\sim &
    \mathcal{D}_d''
} 
\quad\text{for}\quad
d\equiv 2 \;(\bmod\;8),
\end{gathered}
\end{equation}
see \cite{BP20}*{Theorem~1.2 and Section~4.1}.
Now we are ready to prove Case~\ref{case:Gushel-Mukai} of Theorem~\ref{thm:specialK3}.

\begin{cor}
\label{cor:case_spGM}
Let $d = 2g-2 > 6$ and suppose that $d\equiv 2$ or $4\;(\bmod\;8)$, and is not divisible by any prime $p\equiv 3\;(\bmod\;4)$. Then there exists a constant $C>0$ such that 
$$
    \irr\left(\cF_g\right)
    \leq C\cdot g^{10}.
$$
\end{cor}

\begin{proof}
By \eqref{eqn:Fg-Dd}, we have either $\irr(\mathcal{F}_g) = \irr(\mathcal{D}_d)$
if $d\equiv 4\;(\bmod\;8)$, or $\irr(\mathcal{F}_g) = \irr(\mathcal{D}_d') = \irr(\mathcal{D}_d'')$ if $d\equiv 2\;(\bmod\;8)$.
Then the conclusion follows from Corollary~\ref{cor:irratdd}.
\end{proof}

\section{Special hyperk\"{a}hler fourfolds}
\label{sect:periodHKs}

In this section, we start with a brief review on hyperk\"{a}hler fourfolds following the exposition in \cite{DM19}*{Section~3}. We will show that the loci of special hyperk\"{a}hler fourfolds in the period space are components of Heegner divisors, and then derive Theorem~\ref{thm:specialK3}~\ref{case:specialHK} from this.

\subsection{Preliminaries on hyperk\"{a}hler fourfolds}
\label{subsect:preliminary_HK}

Let $(X,H)$ be a smooth polarized hyperk\"{a}hler fourfold of K3-type. The cohomology group $H^2\left(X, \bZ\right)$ is equipped with a $\bZ$-valued bilinear form $q_X$ known as the Beauville--Bogomolov--Fujiki form \cite{Bea83}, which we denote by $\langle , \rangle$ in what follows. The lattice $\left(H^2\left(X,\bZ\right), \langle, \rangle \right)(-1)$ is even of signature $(20,3)$ isometric to 
$$
\Lambda_{\rm HK}
= A_1\oplus E_8^{\oplus 2}\oplus U^{\oplus 3}.
$$
Note that the discriminant group $D(\Lambda_{\rm HK})$ is isomorphic to $\bZ/2\bZ$. Let $h\in H^2(X,\bZ)$ denote the first Chern class of the polarization $H$. We assume further that $h$ is primitive. Then the ideal $\langle h, \Lambda_{\rm HK} \rangle \subset \bZ$ is generated by either $1$ or $2$, which is called the \textit{divisibility} of $h$ and denoted as ${\rm{div}}(h)$.

For a positive integer $n$ and a number $\delta\in\{1,2\}$, the polarized hyperk\"{a}hler fourfolds $(X,H)$ of K3-type with $\langle h,h \rangle =2n$ and ${\rm{div}}(h)=\delta$
form an irreducible quasi-projective moduli space $\mathcal{M}_{2n}^\delta$. This space is nonempty provided that $\delta=1$ and $n$  arbitrary, or $\delta=2$ and $n\equiv 3 \;(\bmod\;4)$ and in both cases it has dimension $20$ \cite{GHS13}.
We denote by $H^2(X,\mathbb{Z})_{\rm prim} := h^{\perp}$ the primitive cohomology lattice. In the case that $\delta=1$, we have 
$$
H^2\left(X,\bZ\right)_{\rm prim}(-1)
\cong\Lambda_{2n}^1
= E_8^{\oplus 2}\oplus U^{\oplus 2}\oplus
\begin{pmatrix}
2 & 0\\
0 & 2n
\end{pmatrix}.
$$
When $\delta=2$ and $n\equiv 3\;(\bmod\;4)$, we have
$$
H^2\left(X,\bZ\right)_{\rm prim}(-1)
\cong\Lambda_{2n}^2
= E_8^{\oplus 2}\oplus U^{\oplus 2}\oplus
\begin{pmatrix}
2 & 1\\
1 & \frac{n+1}{2}
\end{pmatrix},
$$
see \cite{GHS13}*{Example~7.7} or \cite{DM19}*{Section~3.2}. These are even lattices of signature $(20,2)$ with discriminant groups:
\begin{equation}
	\label{eqn:disc_primHK}
	D(\Lambda_{\rm 2n}^1) = \bZ/2\bZ\times\bZ/2n\bZ
	\quad\text{and}\quad
	D(\Lambda_{\rm 2n}^2) = \bZ/n\bZ.
\end{equation}
In any case, there is a period map
$
	\cM_{2n}^\delta\longrightarrow
	\mathcal{P}_{\Lambda_{2n}^\delta}
$
which is an open embedding \cite{Ver13}, see also \cite{GHS13}*{Theorem~2.3 and Remark~2.5}.

\subsection{Special HK-fourfolds and Heegner divisors}
\label{subsect:spHK_Heegner}

A polarized hyperk\"{a}hler fourfold is called {\textit{special of discriminant}} $d$ if $H^{1,1}\left(X,\bZ\right)$ contains a labelling $K$, i.e., a primitive rank $2$ lattice $K$ containing $h$, with ${\rm{disc}}\left(K^{\perp}\right)= -d$.
For every positive integer $d$, we consider the union of hyperplane sections
\begin{equation}
	\label{eqn:unionCnd's}
	\widetilde{\cC}_{n,d}^\delta\colonequals
	\bigcup_{\substack{
			K:\text{ labelling with}\\
			{\rm{disc}}\left(K^\perp\right)=-d
	}}
	K^\perp\cap\Omega(\Lambda_{2n}^\delta)
\end{equation}
and let $\cC_{n,d}^\delta\subset\mathcal{P}_{\Lambda_{2n}^\delta}$ be the corresponding $\Gamma_{\Lambda_{2n}^\delta}$-quotient. This is a divisor in $\mathcal{P}_{\Lambda_{2n}^\delta}$.

\begin{prop}
	\label{prop:HeegnerHK}
	The divisor $\cC_{n,d}^\delta$ is contained in a finite union of Heegner divisors
	$$
	\bigcup_{\gamma\in D(\Lambda_{2n}^\delta)}Y_{\frac{d}{2N}, \gamma}\subset \mathcal{P}_{\Lambda_{2n}^\delta}
	$$
	where $N=4n$ if $\delta=1$ and $N=n$ if $\delta=2$.
\end{prop}

\begin{proof}
	Let $K\subset \Lambda_{\rm HK}$ be a labelling with ${\rm{disc}}\left(K^\perp\right)=-d$. Then the intersection $K\cap h^{\perp}$ is a rank one lattice with generator $\kappa\in \Lambda_{2n}^{\delta}$. Consider the vector
	$$
	v\colonequals
	\frac{1}{{\rm div}_{\Lambda_{2n}^\delta}(\kappa)}\cdot\kappa
	\in(\Lambda_{2n}^\delta)^\vee.
	$$ 
	According to \cite{GHS13}*{Lemma~7.2} and also to the proof of \cite{DM19}*{Proposition~4.1}, one has $d =\langle v, v\rangle\cdot{\rm{disc}}\left(\Lambda_{2n}^\delta\right)$.
	Thus, the union of hyperplanes \eqref{eqn:unionCnd's} is contained in
	$$
	\bigcup_{
		v\in \left(\Lambda_{2n}^\delta\right)^\vee,\;
		\frac{1}{2}\langle v, v\rangle=n
	}
	v^\perp,
	\quad\text{where}\quad
	n = \frac{d}{2\cdot{\rm{disc}}(\Lambda_{2n}^\delta)}.
	$$
	Recall from \eqref{eqn:disc_primHK} that ${\rm{disc}}\left(\Lambda_{2n}^\delta\right)$ is $4n$ if $\delta=1$ and $n$ if $\delta=2$. Taking the $\Gamma_{\Lambda_{2n}^\delta}$-quotient concludes the proof.
\end{proof}



\subsection{Associated moduli spaces of K3 surfaces}
\label{subsect:est_assoK3}

When $\delta = 1$, the moduli spaces $\cF_g$, $\cM_{2n}^1$, and the divisor $\cC_{n,d}^1$ are related in the following way: for a polarized K3 surface $(S,H)$, its Hilbert scheme $S^{[2]}$ of length two subschemes carries a divisor class $H_2 \in H^2(S^{[2]},\mathbb{Z})$ that represents the locus of all schemes which intersect one divisor in $|H|$. Moreover, there is also the class $\Delta$ that represents the locus of non-reduced schemes. This class is divisible by $2$ in ${\rm{Pic}}\left(S^{[2]}\right)$. Assume that $\frac{d}{2}-n=m^2$. Then, as a special case of \cite{DM19}*{Proposition~7.1}, the rational map
\begin{equation}
	\label{eq:specialHK}
	\xymatrix@R=0pt{
		\cF_g\ar@{-->}[r] & \cM_{2n}^1
		: (S,H)\ar@{|->}[r] & \left(S^{[2]}, H_2-\frac{m}{2}\Delta\right),
	}
\end{equation}
after composing with the period map $\mathcal{M}^1_{2n} \longrightarrow \mathcal{P}_{\Lambda^1_{2n}}$, induces a birational isomorphism onto an irreducible component of $\cC_{2n,d}^1$. We use this fact to prove case~\ref{case:specialHK} of Theorem~\ref{thm:specialK3}:

\begin{cor}
	\label{cor:case_spHK}
	Fix a positive integer $n$ and let $d=2g-2$ be such that $\frac{d}{2}-n$ is a square. Then there exists a constant $C_n>0$, depending on $n$,  such that
	$
	\irr\left(\cF_g\right)
	\leq C_n\cdot g^{10}.
	$
\end{cor}

\begin{proof}
	This follows immediately from the previous discussion, together with Proposition \ref{prop:HeegnerHK} and Theorem \ref{cor:heegnerirrationality}.
\end{proof}

\section{A uniform estimate on the degrees of irrationality}
\label{sect:uniformEst}

In this section, we provide a uniform estimate of the degrees of irrationality for the moduli spaces of polarized K3 surfaces. The main ingredient is to realize these moduli spaces as Kudla's special cycles in a certain period space.

\subsection{Siegel modular forms}
\label{subsect:SiegelModForm}

Let $\mathcal{H}_r$ denote the Siegel upper-half space of $r\times r$ symmetric complex matrices with positive definite imaginary part. This has a natural action of the group $\Gamma_r = \operatorname{Sp}_{2r}(\mathbb{Z})$. Let now $k\in \mathbb{Z}$ be an integer and let $V$ be a complex vector space with a representation  $\rho\colon \Gamma_r \longrightarrow \GL(V)$  which is trivial on some congruence subgroup. Then a \emph{Siegel modular form of genus $r$, weight $k$ and type $\rho$} is a holomorphic function $F\colon \mathcal{H}_r \longrightarrow V$ such that
$$
    F(g\cdot \tau) = \det(c\tau+d)^k \rho(g)(F(\tau)), \qquad \text{ for all } g= \begin{pmatrix} a & b \\ c & d \end{pmatrix} \in \Gamma_r.
$$
Such a form has a Fourier series expansion indexed by positive-semidefinite symmetric matrices with rational coefficients:
$$
    F(\tau) = \sum_{T \in \operatorname{Sym}(\mathbb{Q})_{\geq 0}} c_{T}\cdot q^{T} ,\qquad q^T:= e^{2\pi i \operatorname{tr}(T\tau)}.
$$
The Fourier coefficients have at most polynomial growth: more precisely there exists a $C>0$ such that
\begin{equation}
\label{eqn:SiegelBound}
\left|c_T\right|\leq C\cdot {\rm{det}}(T)^{k}
\end{equation}
for all $T$ that are positive definite \cite{AZ95}*{Theorem~3.5}.

\subsection{Kudla's special cycles and their generating series}
\label{subsect:KudlaSpCycles}

We briefly review the generalization of Heegner divisors to higher codimensions due to Kudla \cite{Kud97}. Let $M$ be an even lattice of signature $(m,2)$, with $m$ even and retain the notations in Section~\ref{subsect:Heegner}. To each $r$-tuple $\mathbf{v} = (v_1,\dots,v_r)\in (M^\vee)^{\oplus r}$ we can associate the linear subspace
$
\left<\mathbf{v}\right>
= {\rm{span}}\{v_1,\dots,v_r\}
\subset (M\otimes\bQ)
$
and a \emph{moment matrix}, that is, the $r\times r$ symmetric matrix $
Q(\mathbf{v})\colonequals \frac{1}{2}\left(\langle v_i, v_j\rangle\right)_{ij}$.
At this point, let us fix a matrix $T\in\mathrm{Sym}_r(\bQ)_{\geq0}$ and an $r$-tuple $\boldsymbol{\gamma} \in D(M)^{\oplus r}$. Again $\Gamma_M$ acts on vectors in the equivalence class of ${\bf{\gamma}}$ with fixed moment matrix with finitely many orbits. 
Then the \emph{special cycle of moment $T$ and residue class $\boldsymbol{\gamma}$}, denoted as $Z_{T,\boldsymbol{\gamma}}$, is defined as the image in the period space $\mathcal{P}_M$ of the cycle
$$
\sum_{
	Q(\mathbf{v})= T,\;
	\mathbf{v}\equiv\boldsymbol{\gamma}
}
\langle
\mathbf{v}
\rangle^\perp
\subset\Omega(M).
$$
The $Z_{T,\boldsymbol{\gamma}}$ are cycles in $\mathcal{P}_M$ of codimension equal to the rank of $T$ and in general they are neither reduced nor irreducible. We can collect their classes in the Chow group in  a generating series: first we observe that the restriction of $\cO(1)$ on $\bP(M\otimes\bC)$ to the domain $\Omega(M)$ descends to a line bundle $\mathcal{L}$ on $\mathcal{P}_M$. Then we can form the generating series
$$
\Phi_{\boldsymbol{\gamma}}(\tau)
\colonequals\sum_{T\in\mathrm{Sym}_r(\bQ)_{\geq0}}\;
[Z_{T,\boldsymbol{\gamma}}]\cdot
c_1(\mathcal{L})^{r-\mathrm{rk}(T)}\;
q^T
$$
with coefficients in $\mathrm{CH}^r(\mathcal{P}_M)\otimes \bC$.  There is also a natural Weil representation
$$
    \rho_M\colon \Gamma_r \longrightarrow\GL(\bC[D(M)]^{\oplus r})
$$
which is trivial on a certain congruence subgroup, and  the generating series turns out to be a Siegel modular form with respect to this representation. This was proven for $r=1$ by Kudla \cite{Kud04}*{Theorem~3.2}, and then generalized by Zhang \cite{Zha09}, and also Bruinier and Westerholt-Raum \cite{BW15}*{Theorem~5.2}.


\begin{thm}\cites{Zha09,BW15}
\label{thm:Kudla-Siegel}
For any linear map 
$
	\ell\colon\mathrm{CH}^r(\mathcal{P}_M)\otimes\bC
	\longrightarrow\bC,
$
the series
$$
	\Theta_{\boldsymbol{\gamma}}(\tau)
	\colonequals\ell\left(\Phi_{\boldsymbol{\gamma}}(\tau)\right)
	= \sum_{T\in\mathrm{Sym}_r(\bQ)_{\geq0}}\;
	\ell\left([Z_{T,\boldsymbol{\gamma}}]\cdot
	c_1(\mathcal{L})^{r-\mathrm{rk}(T)}\right)\;
	q^T
$$
is a Siegel modular form of genus $r$, weight $1+\frac{m}{2}$, and type $\rho_M$.
\end{thm}

In particular, we get a result analogous to Theorem \ref{cor:heegnerirrationality} for the irrationality of the special cycles:

\begin{thm}
\label{thm:boundKudlaCycles}
Assume $1\leq r\leq m-2$. For any $\boldsymbol{\gamma}\in D(M)^{\oplus r}$ and $T\in\mathrm{Sym}(\bQ)_{>0}$ of rank $r$, let $Z'_{T,\boldsymbol{\gamma}}$ be an irreducible component of the cycle $Z_{T,\boldsymbol{\gamma}}$. Then there exists a constant $C>0$ such that 
$$
	\irr(Z'_{T,\boldsymbol{\gamma}})\leq
	C\cdot {\rm{det}}\left(T\right)^{1+\frac{m}{2}}.
$$ 
\end{thm}

\begin{proof}
Let $\overline{\mathcal{P}}_{M}$ be the Baily--Borel compactification of the period space $\mathcal{P}_{M}$. Boundary components of $\overline{\mathcal{P}}_M$ correspond to orbits of isotropic subspaces $E\subset M\otimes \bQ$,  see \cite{BB66} and also \cite{GHS07}*{Section 2.2}. Since $M$ has signature $(m,2)$, such a subspace $E$ is at most two dimensional, which implies that the boundary $\overline{\mathcal{P}}_M \setminus \mathcal{P}_M$ is at most one dimensional.
In particular, since $1\leq r\leq m-2$, the localization sequence for Chow groups gives an isomorphism between $\mathrm{CH}^r(\mathcal{P}_M)$ and $\mathrm{CH}^r(\overline{\mathcal{P}}_M)$. This is obtained by associating to each cycle $Z$ in $\mathcal{P}_M$ the class of its closure $\overline{Z}$ in $\overline{\mathcal{P}}_M$.
		
The line bundle $\mathcal{L}$ is ample, so we can fix a multiple $\mathcal{L}^{\otimes k}$ which induces an embedding $\overline{\mathcal{P}}_M \subset  \bP^N$. Moreover, the intersection with $c_1(\mathcal{L})^{m-r}$ induces a homomorphism $\mathrm{CH}^r(\overline{\mathcal{P}}_M) \to \mathbb{Z}$ that we can compose  with the isomorphism $\mathrm{CH}^r(\mathcal{P}_M) \cong \mathrm{CH}^r(\overline{\mathcal{P}}_M)$ to obtain a linear map
	$$
	\ell\colon
	\mathrm{CH}^{r}\left(\mathcal{P}_M\right)\otimes\bC
	\longrightarrow\bC.
	$$
	In particular, if $[Z_{T,\boldsymbol{\gamma}}]$ is a special cycle of codimension $\operatorname{rk}(T)$, we see that 
	\[ \ell([Z_{T,\boldsymbol{\gamma}}]\cdot c_1(\mathcal{L})^{r-\operatorname{rk}(T)}) =  [\overline{Z}_{T,\boldsymbol{\gamma}}]\cdot c_1(\mathcal{L})^{m-\operatorname{rk}(T)} = \frac{1}{k^{m-\operatorname{rk}(T)}} \deg(\overline{Z}_{T,\boldsymbol{\gamma}}) \]
	where $\operatorname{deg}(\overline{Z}_{T,\boldsymbol{\gamma}})$ is the degree of $\overline{Z}_{T,\boldsymbol{\gamma}}$ as a subvariety of $\mathbb{P}^N$. At this point, Theorem~\ref{thm:Kudla-Siegel} implies that the series
	$$
	\sum_{T\in\mathrm{Sym}_r(\bQ)_{\geq0}}\;
	\ell\left([Z_{T,\boldsymbol{\gamma}}]\cdot
	c_1(\mathcal{L})^{r-\mathrm{rk}(T)}\right)\;
	q^T
	$$
	is a Siegel modular form of weight $1+\frac{m}{2}$ for the Weil representation $\rho_M$. Thus, the estimate~\eqref{eqn:SiegelBound} shows that there exists a $C>0$ such that
	\[
	\deg(\overline{Z}_{T,\boldsymbol{\gamma}})
	=k^{m-r}\ell\left([Z_{T,\boldsymbol{\gamma}}]\right)
	\leq C\cdot\det(T)^{1+\frac{m}{2}},
	\]
	for any positive definite $T$. We conclude as in the proof of Theorem~\ref{cor:heegnerirrationality}.
\end{proof}

\subsection{Period spaces of K3 surfaces as special cycles}
\label{subsect:specialK3cycles}

Let $(S,H)$ be a primitively polarized K3 surface of Picard number one and degree $d$, and let $H^2(S,\mathbb{Z})_{\rm prim}$ be its  primitive cohomology, i.e. the orthogonal to the class of $H$. This lattice, if twisted by $(-1)$, is even of signature $(19,2)$, discriminant $d$, and isometric to
$$
H^2(S,\mathbb{Z})_{\rm prim}(-1) \cong \Lambda_{d}\colonequals
E_8^{\oplus 2}\oplus U^{\oplus 2}\oplus \langle d\rangle.
$$
By \cite{Nik80}*{Corollary~1.12.3}, this lattice admits a primitive embedding into an even unimodular lattice of signature $(26,2)$:
\begin{equation}
\label{eqn:lambdaInLambda}
\xymatrix{
	\Lambda_{d}\ar@{^(->}[r] & 
	\Lambda_{\#}\colonequals
	E_8^{\oplus 3}\oplus U^{\oplus 2}.
}
\end{equation}
Let us retain the notations in Section~\ref{subsect:Heegner}. Then the map \eqref{eqn:lambdaInLambda} induces an injection $\Omega(\Lambda_{d})\hookrightarrow\Omega(\Lambda_{\#})$. Since $\Lambda_{\#}$ is unimodular, it also induces an injection
$$\xymatrix{
	\Gamma_{\Lambda_{d}}\ar@{^(->}[r] & 
	\Gamma_{\Lambda_{\#}} : g\ar@{|->}[r] &
	g_{\#}
}$$
where $g_{\#}$ is the unique extension of $g$ that acts trivially on the orthogonal complement $\Lambda_{d}^{\perp\Lambda_{\#}}$ \cite{Huy16}*{Proposition~14.2.6}. Let us denote
$$
\mathcal{P}_{d}\colonequals
\mathcal{P}_{\Lambda_{d}}
= \Omega(\Lambda_{d})/\Gamma_{\Lambda_{d}}
\quad
\text{and}
\quad
\mathcal{P}_{\#}\colonequals
\mathcal{P}_{\Lambda_{\#}}
= \Omega(\Lambda_{\#})/\Gamma_{\Lambda_{\#}}.
$$
The period space $\mathcal{P}_d$ is irreducible and birational to the moduli space $\mathcal{F}_g$.
Then the above injections define a morphism of period spaces
$$\xymatrix{
	f_{d}\colon\mathcal{P}_{d}
	\ar[r] & \mathcal{P}_{\#}.
}$$
By \cite{OS18}*{Lemma~4.3}, if two K3 surfaces (possibly with polarizations of different degrees) have the same period point in $\mathcal{P}_{\#}$, then their transcendental lattices are Hodge isometric, which is equivalent to being Fourier--Mukai partners \cite{Orl97}*{Theorem~3.3}. In general, a given point in $\mathcal{P}_{\#}$ may have infinitely many preimages in $\bigcup_{d\geq 2}\mathcal{P}_d$. But if a K3 surface has Picard number one, then all of its Fourier--Mukai partners are of Picard number one with the same polarization degree. This allows us to use the counting formula in \cite{HLOY03} for such partners of a fixed degree and conclude that
\begin{equation}
\label{eqn:FM-number}
	\deg(f_d) \leq 2^{\omega(g-1)-1}
\end{equation}
where $d=2g-2$ and $\omega(g-1)$ is the number of distinct prime factors of $\frac{d}{2}=g-1$.

\begin{rmk}
The paper \cite{OS18} is written in the context of Shimura varieties, where the period spaces $\mathcal{P}_d$ and $\mathcal{P}_{\#}$ are defined instead as quotients of $\Omega(\Lambda_d)$ and $\Omega(\Lambda_{\#})$ by the arithmetic groups 
$$
	{\rm{SO}}(\Lambda_d)\cap \Gamma_{\Lambda_d}
	\quad\text{and}\quad
	{\rm{SO}}(\Lambda_{\#})\cap \Gamma_{\Lambda_{\#}}.
$$
These are subgroups of the monodromy groups $\Gamma_{\Lambda_{d}}$ and $\Gamma_{\Lambda_{\#}}$ of index $2$, so the period spaces of \cite{OS18} are double covers of the period spaces considered here.
\end{rmk}

Notice that the discriminant group $D(\Lambda_{\#})$ is trivial as $\Lambda_{\#}$ is unimodular. In particular, the special cycles $Z_{T,\boldsymbol{\gamma}} = Z_{T,\mathbf{0}}$ are indexed only by $T\in\mathrm{Sym}_r(\bQ)_{\geq0}$.

\begin{prop}
	\label{prop:ImageKudla}
	For each $d$, the image $f_d(\mathcal{P}_{d})\subset\mathcal{P}_{\#}$ is an irreducible component of the special cycle $Z_{T,\mathbf{0}}$ for some $T\in\mathrm{Sym}_7(\bQ)_{>0}$ with
	$$
	\det(T) = \frac{1}{2^7}\cdot d.
	$$
\end{prop}

\begin{proof}
	Under the embedding defined in \eqref{eqn:lambdaInLambda}, let us fix a basis $ \mathbf{w} = (w_1,\dots,w_7) \in\Lambda_{\#}^{\oplus7}$
	for the orthogonal complement $\Lambda_d^{\perp \Lambda_{\#}}$ and let $T\colonequals\frac{1}{2}(\langle w_i,w_j\rangle)$
	be the corresponding moment matrix. By construction,
	$$
	\Omega\left(\Lambda_d\right)
	= \langle\mathbf{w}\rangle^{\perp}
	\subset\Omega\left(\Lambda_{\#}\right).
	$$
	Therefore,
	$$
	\Omega\left(\Lambda_d\right)
	\subset
	\sum_{Q(\mathbf{v})=T}\langle\mathbf{v}\rangle^\perp.
	$$
	This shows that $f_d(\mathcal{P}_{d})\subset Z_{T,\mathbf{0}}$ and since they have the same dimension, it must be an irreducible component. Finally, using the fact that $\Lambda_{\#}$ is unimodular, we conclude that
	$$
	\det(T)
	=\frac{1}{2^7}\cdot{\rm{disc}}(\Lambda_d^{\perp \Lambda_{\#}})
	=\frac{1}{2^7}\cdot{\rm{disc}}(\Lambda_d)
	=\frac{1}{2^7}\cdot d,
	$$
\end{proof}

\subsection{The uniform estimate}
\label{subsect:uniformEst}

Before proving Theorem~\ref{thm:allK3}, let us recall that $\omega(n)$ counts the number of distinct prime factors of a positive integer $n$. Then $2^{\omega(n)} \leq d(n)$ where $d(n)$ is the divisor function that counts the number of divisors of $n$. It is well known, that
\begin{equation}
	\label{eqn:est_on_FM}
	2^{\omega(n)} \leq d(n)
	= O(n^\varepsilon)
	\quad\text{for every}\quad
	\varepsilon>0
\end{equation}
see, for example, \cite{HW08}*{Sections~18.1, 22.11, and 22.13}. Now we are ready to give the uniform estimate:

\begin{thm}[= Theorem~\ref{thm:allK3}]
	\label{thm:allK3_2}
	For each $\varepsilon>0$, there exists a constant $C_\varepsilon>0$ such that
	$
	\irr(\mathcal{F}_g)
	\leq C_\varepsilon\cdot g^{14+\varepsilon}
	$
	for every $g$.
\end{thm}

\begin{proof}
	Recall that $d=2g-2$ and that $\Lambda_{\#}$ has signature $(26,2)$. These facts, together with Theorem~\ref{thm:boundKudlaCycles} and Proposition~\ref{prop:ImageKudla}, imply that there exists a constant $C>0$ independent of $g$ such that
	$\irr(f_d(\mathcal{P}_d))
	\leq C\cdot g^{14}$.
	On the other hand, since $\mathcal{F}_g$ and $\mathcal{P}_d$ are birational, it follows from \eqref{eqn:FM-number} that
	$
	\irr(\mathcal{F}_g)
	\leq 2^{\omega(g-1)}\cdot
	\irr(f_d(\mathcal{P}_d)).
	$
	Combining the above two inequalities, we obtain
	$
	\irr(\mathcal{F}_g)
	\leq 2^{\omega(g-1)}\cdot C\cdot g^{14}.
	$
	Applying \eqref{eqn:est_on_FM} to the factor $2^{\omega(g-1)}$, we conclude that, for every $\varepsilon>0$, there exists a constant $C_\varepsilon>0$ such that
	$
	\irr(\mathcal{F}_g)
	\leq C_\varepsilon\cdot g^{14+\varepsilon}.
	$
\end{proof}

\begin{rmk}
Our motivation comes from understanding the birational complexity of the family of moduli spaces $\mathcal{F}_g$, but one observes that this strategy can be used to bound the degree of irrationality of a much larger class of Shimura varieties of orthogonal type. Whenever one has an infinite family of even lattices $\Lambda_n$ of signature $(m,2)$, by taking an appropriate number of copies of $E_8$, the lattices $\Lambda_n$ can all be primitively embedded in $\Lambda_\#=U^{\oplus 2}\oplus E_8^{\oplus \ell}$ where the action of the arithmetic group on $\Omega(\Lambda_n)$ extends to a suitable action on $\Omega(\Lambda_\#)$. This is indeed the case for families of moduli spaces of abelian surfaces, hyperk\"{a}hler varieties of Kummer type and K3-type. It is a subtle problem to understand or bound the degree of the induced map on period spaces
$
	f_n:\mathcal{P}_n\longrightarrow\mathcal{P}_\#.
$
\end{rmk}

\subsection{A concrete construction}
\label{sec:ogradyexamples}

In our method, the maps $\mathcal{F}_g\dashrightarrow \mathbb{P}^{19}$ giving the bound on the degree of irrationality are constructed indirectly. However, a construction of O'Grady  provides concrete examples of maps satisfying the bounds in Theorem~\ref{thm:specialK3}. More precisely, O'Grady constructs finite surjective maps
\begin{equation}
\label{eq:ogradymaps}
    f_{n,k}\colon \mathcal{F}_{n^2k+1} \longrightarrow \mathcal{F}_{k+1}.
\end{equation}
These maps are defined in \cite{OGr89}*{page~163} using lattices, which can also be interpreted via moduli spaces of twisted K3 surfaces \cite{DM21}*{Example~4.15}: the space $\mathcal{F}_{n^2k+1}$ is birational to the moduli space of K3 surfaces of genus $k+1$, together with a Brauer class of order $n$, and then the map \eqref{eq:ogradymaps} simply forgets the Brauer class.

The maps \eqref{eq:ogradymaps} imply that $\irr(\mathcal{F}_{n^2k+1}) \leq \irr(\mathcal{F}_{k+1})\cdot\deg(f_{n,k})$. Hence, if we can bound $\deg(f_{n,k})$ in terms of $n$, we obtain an analogous bound for the irrationality of $\mathcal{F}_{n^2k+1}$. This was done by Kond\={o}, in the case that $k=1$ and $n=p$ is an odd prime number. He proved in \cite{Kon93}*{Lemma~3.2}  that $\deg(f_{p,1}) = p^{20}+p^{10}$, so that
$$
    \irr(\mathcal{F}_{p^2+1})
    \leq \irr(\mathcal{F}_2)\cdot (p^{20}+p^{10}).
$$
Observe that this is exactly the same kind of bound appearing in Theorem~\ref{thm:specialK3}.

\bibliography{IrrationalK3_bib}
\bibliographystyle{alpha}

\ContactInfo
\end{document}